\documentclass[12pt,leqno,fleqn]{amsart}  
\usepackage{amsmath,amstext,amsthm,amssymb,amsxtra}

\usepackage{mathtools}
\mathtoolsset{showonlyrefs,showmanualtags} 

\usepackage{float}

\usepackage{hyperref} 
\hypersetup{
    colorlinks=true,       
    linkcolor=blue,          
    citecolor=magenta,        
    filecolor=magenta,      
    urlcolor=cyan,          
}

\usepackage[backrefs]{amsrefs}

\usepackage{pgf,pgfarrows} 
\usepackage{tikz}
\usetikzlibrary[decorations.pathreplacing]

\allowdisplaybreaks
\swapnumbers

\allowdisplaybreaks
\swapnumbers

\theoremstyle{plain} 
\newtheorem{lemma}[equation]{Lemma} 
\newtheorem{proposition}[equation]{Proposition} 
\newtheorem{theorem}[equation]{Theorem} 
\newtheorem{corollary}[equation]{Corollary} 
\newtheorem{conjecture}[equation]{Conjecture}

\theoremstyle{definition}
\newtheorem{definition}[equation]{Definition} 

\theoremstyle{remark}
\newtheorem{remark}[equation]{Remark}
\newtheorem*{ack}{Acknowledgment}

\numberwithin{equation}{section}

\def\norm#1.#2.{\lVert#1\rVert_{#2}}
\def\Norm#1.#2.{\bigl\lVert#1\bigr\rVert_{#2}}
\def\NOrm#1.#2.{\Bigl\lVert#1\Bigr\rVert_{#2}}
\def\NORm#1.#2.{\biggl\lVert#1\biggr\rVert_{#2}}
\def\NORM#1.#2.{\Biggl\lVert#1\Biggr\rVert_{#2}}

\def\ip#1,#2,{\langle #1,#2\rangle}
\def\Ip#1,#2,{\langle#1,#2\rangle}
\def\IP#1,#2,{\langle#1,#2\rangle}

\def\mid{\,:\,}

\def\XXint#1#2#3{{\setbox0=\hbox{$#1{#2#3}{\int}$}
     \vcenter{\hbox{$#2#3$}}\kern-.5\wd0}}

\title {On Muckenhoupt-Wheeden Conjecture}

\author[M.C. Reguera]{Maria Carmen Reguera}
\thanks{Research supported in part by NSF grant 0456611}
\address{ School of Mathematics, Georgia Institute of Technology, Atlanta GA 30332, USA}
\email {mreguera@math.gatech.edu}

\begin{document}

\begin{abstract}
Let $ M$ denote  the dyadic Maximal Function.  
We show that there is a weight $ w$,  and  Haar multiplier $ T$ for which the following  weak-type inequality fails.  
$$
\sup_{t>0}t w\left(\left\{  x\in\mathbb R \mid |Tf(x)|>t\right\}\right)\le C \int_{\mathbb R}|f|Mw(x)dx.
$$ 
(With $ T$ replaced by $ M$, this is a well-known fact.) This shows that a dyadic version of 
the so-called Muckenhoupt-Wheeden Conjecture is false.  This accomplished by  using current techniques 
in weighted inequalities to show that a particular $ L ^2 $ consequence of the inequality above does not hold.  
\end{abstract}

\maketitle

\section{Introduction} 


The starting point of this work goes back to 1971 \cite{MR0284802}, when C. Fefferman and E. Stein, 
in order to establish vector-valued estimates for the maximal function,
proved that if $w$ is a weight, namely a non-negative locally integrable function, 
and $M$ denotes the Hardy-Littlewood maximal operator then
\begin{equation}
\sup_{t>0}t w\left(\left\{ x\in \mathbb R^{n}\mid Mf(x)>t\right\}\right)\leq c\int_{\mathbb R^{n}}|f|Mw(x).
\end{equation}


A very natural question was then raised by B. Muckenhoupt and R. Wheeden (see \cite{MR2427454}):
could we replace the Hardy-Littlewood maximal operator $M$ by a Calder\'on-Zygmund operator? Their conjecture, known as Muckenhoupt-Wheeden Conjecture, is stated below.

\begin{conjecture} (Muckenhoupt-Wheeden) Let $w$ be a weight and $M$ be the Hardy-Littlewood maximal operator. Let $T$ be a Calder\'on-Zygmund operator with $\lVert T \rVert_{CZO}\le 1$. Then
\begin{equation}\label{e.MWC}
\sup_{t>0}t w\left(\left\{  x\in\mathbb R \mid |Tf(x)|>t\right\}\right)\le C \int_{\mathbb R}|f|Mw(x)dx
\end{equation}
\end{conjecture}

The exact definition of Calder\'on-Zygmund operator need not concern us here, though it certainly includes the non-positive Hilbert transform (see chapter VII of \cite{MR1232192} for precise definitions). The hope was that the Conjecture identified a somewhat robust principle.  We herein 
disprove the \emph{dyadic version} of this conjecture.  So, $ M$ is replaced by the (smaller) dyadic maximal 
function, and $ T$ will be a Haar multiplier, which are the simplest possible dyadic Calder\'on-Zygmund operators.


Endpoint estimates are known to be the most delicate ones, and very frequently they are also the most powerful.
That is the case of Muckenhoupt-Wheeden  Conjecture. For instance,  
an extrapolation result due to D. Cruz-Uribe and C. P\'erez \cite{MR1761362} shows this: 
If $w$ is a weight and \eqref{e.MWC} holds with $T$ a sublinear operator then
\begin{equation}\label{e.extrap}
\int_{\mathbb R} |T(f)|^{p}w(x)dx\le \int_{\mathbb R} |f|^{p}\left(\frac{Mw}{w}\right)^{p}w(x)dx.
\end{equation}
The dyadic version of this result is also true (see \cite{MR1761362}, Remark 1.5).

With a few partial results that we shall discuss later in the introduction,
the Muckenhoupt-Wheeden Conjecture has been open up to today's date.
In this paper, we answer the dyadic version of 
\eqref{e.MWC} in the negative by disproving \eqref{e.extrap}.  	
We are ready to state our main theorem.


\begin{theorem}\label{t.main}
There exist a weight $w$ and a Haar multiplier $T$ for which $T$ is unbounded as map from  $L^{2}\left((\frac{Mw}{w})^{2}w\right)$ to $L^{2}(w)$.
\end{theorem}

As a corollary we solve a long-standing conjecture,

\begin{corollary}\label{c.mwfalse}
Muckenhoupt-Wheeden conjecture in its dyadic version is false.
\end{corollary}



For the proof we construct a measure $w$ and a Haar multiplier $T$ that avoids all cancellations. The tool behind this construction is the corona decomposition, that has proven to be very useful in finding sharp estimates when the weight is in the $A_{p}$ class \cites{0906.1941,1006.2530,1007.4330,1010.0755}.


Throughout the literature, there has been evidence for a positive answer to the conjecture as well as for a negative one.
S. Chanillo and R. Wheeden \cite{MR891775} showed that 
a square function satisfied the Muckenhoupt-Wheeden Conjecture.
We also mention the work of Buckley \cite{MR1124164}, who  in dimension $n$ proved that \eqref{e.MWC} holds for weights $w_{\delta}(x)=|x|^{-n(1-\delta)}$ for $0<\delta<1$. 

The sharpest results in this direction are due to C. P\'erez \cite{MR1260114}: If $T$ is a Calder\'on-Zygmund operator and $M^{2}$ is the Hardy-Littlewood maximal operator iterated $2$ times, 
$$
\sup_{t>0}t w\left (\left\{  x\in\mathbb R \mid |Tf(x)|>t\right\}\right )\le C \int_{\mathbb R}|f|M^{2}w(x)dx.
$$
He actually proved something better, $M^{2}$ can be replaced by the smaller operator $M_{L(logL)^{\epsilon}}$.
In an attempt to understand these endpoint estimates, A. Lerner, S. Ombrosy and C. P\'erez considered a somehow "dual" problem of Muckenhoupt-Wheeden, we refer the reader to \cite{MR2511869}.
A negative answer to \eqref{e.MWC} was provided by M.J. Carro, C. P\'erez, F. Soria and J. Soria when T is a fractional integral, \cite{MR2151228}. There are two points that distinguish this example from the singular integral one: 1) the lack of cancellation when treating positive operators and 2) the construction depends upon T being a true fractional integral and does not 
allow an immediate extension to the singular integral case.


By imposing an extra condition on the weight $w$, a weaker version of Muckenhoupt-Wheeden can be formulated. This is known as the Weak Muckenhoupt-Wheeden Conjecture and appears in work of A. Lerner, S. Ombrosy and C. P\'erez \cite{MR2427454}, \cite{MR2480568}. 

\begin{conjecture} (Weak Muckenhoupt-Wheeden) Let $w$ be an $A_{1}$ weight and let $\lVert w \rVert_{A_{1}}$ be the $A_{1}$ constant associated to it. Let $T$ be a Calder\'on-Zygmund operator with $\lVert T \rVert_{CZO}\le 1$. Then
\begin{equation}\label{e.wMWC}
\sup_{t>0}t w\left(\left\{  x\in\mathbb R \mid |Tf(x)|>t\right\}\right)\le C \lVert w \rVert_{A_{1}}\int_{\mathbb R}|f|w(x)dx
\end{equation}
\end{conjecture}

Recall that $w$ is an $A_{1}$ weight if there exists a constant $c>0$ such that $Mw(x)\leq cw(x)$ a.e. The smallest of such constants 
$c$ is denoted by $\lVert w \rVert_{A_{1}}$.  Thus, the Weak Muckenhoupt-Wheeden Conjecture would be 
an immediate consequence of \eqref{e.MWC}, were it true. 
The continuity of Calder\'on-Zygmund operators in $L^{1}\mapsto L^{1,\infty}$ when $w$ is an $A_{1}$ weight is well known and goes back to the origins of the weighted theory with R. Hunt, B. Muckenhoupt and R. Wheeden \cite{MR0312139} in dimension $1$ and R. Coifman and C. Fefferman in higher dimensions \cite{MR0358205}.  
The novelty of \eqref{e.wMWC} resides with
 the linear dependence on $\lVert w \rVert_{A_{1}}$. Linear growth of the $A_{1}$ constant has been proven in the strong case for $p>1$, that is
$$
\norm T.L^{p}(w)\mapsto L^{p}(w).\leq C(p,n)\lVert w \rVert_{A_{1}},
$$
and this is sharp. The result was first proven by R. Fefferman and J. Pipher for $p\geq 2$ and $T$ a classical singular integral
operator in \cite{MR1439553}. Later on it was extended to $p>1$ and general Calder\'on-Zygmund operators by
A. Lerner, S. Ombrosy and C. P\'erez in \cite{MR2427454}. The proof of A. Lerner \emph{et al.}
provides not only  linear dependence on the $A_{1}$ constant, but also explicit
dependence of the operator norm on $p$. The explicit dependence on $p$  allows one to get the weak endpoint below,
$$
\sup_{t>0}t w\left(\left\{  x\in\mathbb R \mid |Tf(x)|>t\right\}\right)\le C \lVert w \rVert_{A_{1}}(1+\log\lVert w \rVert_{A_{1}})\int_{\mathbb R}|f|w(x)dx.
$$

We want to point out that even though this estimate is far from proving \eqref{e.wMWC}, it is the best result known up to this date. The Weak Muckenhoupt-Wheeden Conjecture remains open and we will not make any new contribution
to it in this paper, but we are hoping to shed some light in the understanding of these endpoint estimates. 


The paper is organized as follows. In Section 2 we introduce the necessary concepts.
Section 3 presents an important tool in the proof of the result, the corona decomposition associated to a measure $w$. We dedicate Section 4 to the inductive construction of measures $w$ and Haar multipliers $T$. Section 5 includes the proof of the main theorem. Finally, the last Section contains the bibliography. 

\begin{ack}
The author would like to thank her dissertation adviser, Michael Lacey, for his support and suggestions in the composition of this paper, and Carlos P\'erez for many interesting conversations on the topic.
\end{ack}

\section{Basic concepts} 

The space we will be working on is $\mathbb R$. Throughout the paper $|\cdot |$ will stand for the Lebesgue measure in $\mathbb R$, 
$1_{E}$ will be the characteristic function associated to the set $E\subset \mathbb R$,  and for   $ x\ge 0$, $[x]$ denotes the 
integer part. The letters $i,j,l,k$ will stand for positive integers. C will denote a universal constant, not necessarily the same in each case.

In the sequel when referring to $M$ we will understand the dyadic maximal function, i.e., for $f\in L^{1}_{loc}$
$$
Mf(x)=\sup_{Q\,\, dyadic}\frac{1_{Q}}{|Q|} \int_{Q} f(x)dx.
$$

We will use a different formulation of the two weight inequality \eqref{e.extrap}. This characterization was first introduced by E. Sawyer in \cite{MR676801} and have been used since then, becoming one of the standard approaches. The proof is a well known exercise that we are not including in this paper.

\begin{proposition}\label{p.dual} Let $w,v$ be two positive Borel measures, continuous with respect to Lebesgue measure and let $T$ be a sublinear operator. Let $C$ be a universal constant and $1<p<\infty$, the statements below are equivalent,
\begin{align}
\label{e.general} \norm Tf.L^{p}(w).& \leq  C\norm f.L^{p}(v).,\\
\label{e.Sawyer} \norm T(f\sigma).L^{p}(w).& \leq  C\norm f.L^{p}(\sigma).,\quad \sigma=v^{1-p'}1_{\textup{supp}(v)}.\\
\end{align}
\end{proposition}

 \begin{remark}
This new formulation provides a more symmetric estimate for $T^{*}$, the dual operator with respect to Lebesgue measure, i.e., \eqref{e.Sawyer} is equivalent to 
\begin{equation}\label{e.dual}
\norm T^{*}(fw).L^{p'}(\sigma). \leq  C\norm f.L^{p'}(w).,
\end{equation}
where $\frac{1}{p}+\frac{1}{p'}=1$.
\end{remark} 

The estimate we want to disprove is a particular case of a 2-weighted inequality. We will work with weights $w\geq 0$, $w\in
L^{1}_{\textup{loc}}$ and $\sigma= \frac{w}{(Mw)^{2}}$, which is the dual measure of $v=\left( \frac{Mw}{w}\right)^2 w$. Throughout the
paper $\sigma$ will take the above form. The operators we consider are discrete dyadic
operators. Let us recall some of the basic concepts associated to them before getting to the precise definition.

\begin{definition}\label{d.haar} Let $\mathcal D$ be the usual dyadic grid in $\mathbf R$, namely \\
$\mathcal D = \{ \left.\left [ 2^{k}m, 2^{k}(m+1)\right.\right ),\,\,\,m, k \in \mathbf Z \}$. Let $I=[a,b]$ be an
interval in $\mathcal D$, then $I^{-}=[a, \frac{a+b}{2})$ is the \emph{left child of I}
and $I^{+}=[\frac{a+b}{2}, b)$ is the \emph{right child of I}.
We define the $L^{2}$-normalized Haar function associated to $I$, $h_{I}$ as
$$ h_{I}= \frac{1_{I^{+}}-1_{I^{-}}}{|I|^{1/2}}$$
\end{definition}

Our interest will lay on particular examples of dyadic operators, the Haar multipliers. 

\begin{definition} Let $\epsilon=(\epsilon_{I})_{I\in \mathcal D}$ be a bounded sequence. The operator $T_{\epsilon}$ is a \emph{Haar multiplier} associated to $\epsilon$ iff 
\begin{equation}\label{e.haarmultiplier}
 T _{\epsilon}  f = \sum _{I\in \mathcal D} 
 \epsilon_{I}\ip f, h_{I}, h_{I}.
\end{equation}
\end{definition}

\section{The Corona Decomposition}

This Section provides the tool to, given $w$, decompose measure $\sigma$. At the same time, the corona decomposition allows to group the dyadic intervals into families and consequently decompose any dyadic operator into a sum of operators. This was previously done in \cite{0906.1941}. 

\begin{definition}\label{d.linearizing} Let $ \mathcal D ' \subset  \mathcal D$ be any 
 collection of dyadic cubes.  Call $ (\mathcal L \;:\; \mathcal D' (L) ) $ 
a \emph{corona decomposition of $ \mathcal D '$ relative to measure $ w$} 
 if these conditions are met. 
Let $L, L' \in \mathcal L$,  we have 
\begin{equation}\label{e.lin1}
\frac {w (L')} {\lvert  L'\rvert } \geq 4 \frac {w (L)} {\lvert  L\rvert } \,,  
\qquad  L'\subsetneq L . 
\end{equation} 
Define $ \Gamma \;:\; \mathcal   D' \to \mathcal L $ by requiring that $ \Gamma (I) $ be the minimal element of 
$ \mathcal L$ that contains $ I$.  We set  $ \mathcal D' (L) := \{ I\in \mathcal D' \;:\; \Gamma (I)= L\}$. Then for all $I\in  \mathcal D' (L)$ we
 require 
\begin{equation}\label{e.lin2}
4\frac {w (L)} {\lvert  L\rvert } >   \frac {w (I)} {\lvert  I\rvert }  \,. 
\end{equation} 
\end{definition}

\begin{remark}
The corona decomposition is obtained by a stopping time argument. It is for this reason that we will refer to $\mathcal L$ as the \emph{stopping collection} in the corona decomposition.
\end{remark}	

The collections $ \mathcal D' (L)$ partition $ \mathcal D'$. Since decompositions of dyadic intervals correspond directly to decompositions of dyadic operators, we can write any Haar multiplier as
\begin{equation} \label{r.decompT}
T_{\epsilon}= \sum_{L\in \mathcal L} T_{L} \quad {\text{where}} \quad 
T_{L}= \sum_{I\in \mathcal D(L)} \epsilon_{I}\ip f, h_{I},  h_{I}.
\end{equation}

We now focus on the structure of the measure $\sigma$. We will denote $\mathcal L_{0}$ to be the set of maximal intervals in $\mathcal L$. In general, we denote $\mathcal L_{j}$ to be the maximal intervals on $ \mathcal L\backslash \bigcup_{i=0}^{j-1}\mathcal L_{i}$.

\begin{definition}
Let $E\in \mathcal D(L)$, $L\in \mathcal L_{j}$, we define the \emph{take away the children} operator on sets of $\mathbb R$ as 
$\displaystyle \Delta_{1} E= E\backslash \bigcup_{L'\in \mathcal L_{j+1},\, L'\subset E} L'$. In general, we define 
$$
\Delta_{l} E= \cup_{\tilde{L}\in \mathcal L_{j+l-1},\,\tilde{L}\subset E} \tilde{L} \backslash \cup_{L'\in \mathcal L_{j+l},\, L'\subset E} L'\quad \text{for } \, l> 1.
$$
\end{definition}

\begin{remark}\label{decompM}
 This last definition helps us track the value of $ M w$.   
Notice that for every $x \in \Delta_{l} E$, $E\in\mathcal D(L)$, and  $L\in\mathcal L$,   we have  
$\displaystyle 8^{l}\frac{w(L)}{|L|}\geq Mw(x)\geq 4^{l}\frac{w(L)}{|L|}$. Since $\{\Delta_{l} E \}_{l\geq 1}$ forms a partition of $E$, we can estimate 
\begin{equation}\label{decompsigma}
\sigma(E)=\int_{E}\frac{w}{Mw^{2}}(x)dx \geq \sum^{\infty}_{l=1} 8^{-2l}\left(\frac{|L|}{w(L)}\right)^{2}w(\Delta_{l}E)
\end{equation}
\end{remark}

\section{The inductive construction} 

In this Section, we describe the inductive procedure that will provide measures $w_{k}$ and operators $T_{k}$, the key elements in proving Theorem \ref{t.main}. 
We start with a few definitions associated to the base case.

\begin{definition}\label{d.jumpingpoint}
Let $J=\left.\left [ a, a+\alpha\right)\right.$ be a dyadic interval, we define the \emph{jumping point} of $J$, and we denote it by $\textup{jp}(J)$, as
$\textup{jp}(J):=a+\frac{\alpha}{3}$. We also denote the right end point of $J$ by $\textup{rep}(J):=a+\alpha$. 
\end{definition}

Notice that the ``jumping point'' divides the interval into two intervals of lengths one-third and two-thirds the length 
of the original interval, thus  
  the ``jumping point'' has  
 a periodic binary expansion, which fact is important to the construction. 
We now define the measure that gives name to the jumping point.
\begin{definition}\label{d.measure}
Let $J$ be a dyadic interval and $\lambda>0$ be a height, we define the measure associated to $J$ with height $\lambda$ as
$$
\mu_{J}^{\lambda}= \lambda 1_{\left.\left [ \textup{jp}(J), \textup{rep}(J)\right)\right.}.
$$
\end{definition}

Having  listed the key ingredients to construct our measure $w_{k}$ inductively, we now focus on those associated to the construction of the operator $ T^{k}$.

\begin{definition}\label{d.haarmult}    Let $k\geq 1$ be a fixed integer.  
Let $J$ be a dyadic interval, we define $\Xi_{J}$  as the following collection of intervals associated to $J$,
\begin{equation}\label{e.int}
\Xi_{J}:=\left \{ J=I_{0}\supseteq \ldots \supseteq I_{2k} \mid \textup{jp}(J)\in I_{i}^{-}, \, |I_{i}|=4|I_{i+1}| \, \right\}.
\end{equation}
We denote $I(J):=I_{2k}$, the minimal interval in the collection $\Xi_{J}$. And we define the collection of the right children of the intervals in $\Xi_{J}$ as
\begin{equation}\label{e.int+}
\Xi_{J}^{+}:=\left \{ I_{i}^{+} \mid I_{i}\in \Xi_{J} \setminus I(J)\right\}. 
\end{equation}
We are now ready to define the Haar multiplier associated to $J$ with sign $r_J$, $S_{J, r(J)}$, as
\begin{equation}\label{e.haarmul}
S_{J, r(J)}(f):=r_{J}\sum_{I\in \Xi_{J}} \langle f, h_{I} \rangle h_{I}. 
\end{equation}
where $r_{J}\in \{+1,-1\}$.
\end{definition}

\begin{remark}
There are two ideas about the jumping point that should be clarified. (1) If $\textup{jp}(J)\in I$, then either
$\textup{jp}(J)\in I^{-}$ or $\textup{jp}(J)\in I^{+}$. Moreover, these two events alternate, i.e., let $I\subset I'\subset J$ with $|I|=\tfrac12|I'|$ and $\textup{jp}(J)\in I, I'$. 
We have that if $\textup{jp}(J)\in I'^{-}$ (respectively $I'^{+}$) then $\textup{jp}(J)\in I^{+}$ (respectively $I^{-}$).
This phenomenon explains that the chosen intervals in \eqref{e.int} satisfy $|I_{i}|=4|I_{i+1}|$. 
 (This is the consequence of the binary expansion of $ 1/3$.) 
(2) We take advantage of the localization of the jumping point in another way. The intervals $I\in \Xi_{J}^{+}$ ``almost'' form a partition of
the support of $\mu_{J}^{\lambda}$. If $k\rightarrow\infty$, they will actually form a partition,  since  we consider only a fixed number of them, we can only get
\begin{equation}\label{e.supp}
{\text{ supp}} \mu_{J}^{\lambda}= \left[\left. \textup{jp}(J), \textup{rep}(I(J))\right)\right.\cup\bigcup_{I\in \Xi_{J}^{+}} I.
\end{equation}

\end{remark}

\begin{remark}
Notice that $\Xi_{J}$ and consequently $\Xi_{J}^{+}$ and $S_{J, r(J)}$ depend on the parameter $k$, that will play the role of the induction index in the proof of Proposition \ref{e.mainestimate}. For the sake of simplicity, we omit the parameter $k$ in the notation of those objects.
\end{remark}


 The following lemma takes advantage of the lack of cancellation in $S_{J, r(J)}(\mu_{J}^{\lambda})$ to compare the $\mu_{J}^{\lambda}$ measure of a  level set associated to $S_{J, r(J)}$ with that of $J$. 

\begin{lemma}\label{l.distestimate}
Let $k\geq 1$ be a fixed integer, $J$ be a dyadic interval and $\mu_{J}^{\lambda}$ and $S_{J, r(J)}$ as above. Then, 
first the inner products $\langle \mu_{J}^{\lambda}, h_{I_{i}}\rangle $, for $ I_i\in \Xi_J$ depend only on the 
numbers $ \mu _{J} ^{\lambda } (I (J))$ and $ \{ \mu _J ^{\lambda } (I) \;:\; I\in \Xi _{J} ^{+}\}$.  Second, 
\begin{equation}\label{e.distestimate}
\mu_{J}^{\lambda}\left(\left \{ x \mid |S_{J, r(J)}\mu_{J}^{\lambda}(x)|> k\frac{\mu_{J}^{\lambda}(J)}{|J|}\right \}\right)\geq \frac{1}{4}2^{-4k}\mu_{J}^{\lambda}(J).
\end{equation}
\end{lemma}

\begin{proof}
For the first claim, $ \langle \mu_{J}^{\lambda}, h_{I_{i}}\rangle $ depends 
only on the measure $ \mu _{J} ^{\lambda }$ assigned to the two children for $ I_i$.  And, these 
two children are unions of the sets in \eqref{e.supp}.

Turning to the second claim, 
let $J$ be a dyadic interval, for any $I_{i}\in \Xi_{J}$ we have the following   equality 
\begin{equation}\label{e.onethird}
\frac{\langle \mu_{J}^{\lambda}, h_{I_{i}}\rangle}{\sqrt{|I_{i}|}}= \frac{1/2\lambda |I_{i}|- \left(1/2-1/3\right)\lambda |I_{i}|}{|I_{i}|}= \frac{\lambda}{3}.
\end{equation}

Since $I_{i+1}\subset I_{i}^{-}$ for all $i$, $\langle \mu_{J}^{\lambda}, h_{I_{i}}\rangle h_{I_{i}}$ is constant on $I_{i+1}$. Therefore, using \eqref{e.onethird} for every $x\in I(J)^{-}$
$$  
|S_{J, r(J)}\mu_{J}^{\lambda}(x)|=\sum_{i=0}^{2k} \frac{\langle \mu_{J}^{\lambda}, h_{I_{i}}\rangle}{\sqrt{|I_{i}|}}1_{I(J)^{-}}= \frac{(2k+1)\lambda}{3}.
$$
On the other hand, 
$$
\frac{\mu_{J}^{\lambda}(J)}{|J|}= \frac{\frac23 \lambda |J|}{|J|}= \tfrac{2}{3}\lambda,
$$
and $\displaystyle \frac{(2k+1)\lambda}{3}> \frac{2}{3}k\lambda$ trivially. 
This added to the fact that $|S_{J, r(J)}\mu_{J}^{\lambda}(x)|\leq \frac{2}{3}k\lambda$ for all $x\in J$ and $x\notin I(L)^{-}$ gives,
\begin{equation}\label{e.mineqlevelset}
I(L)^{-}= \left \{ x \mid |S_{J, r(J)}\mu_{J}^{\lambda}(x)|> k\frac{\mu_{J}^{\lambda}(J)}{|J|}\right \}, 
\end{equation}
and    
$$
\mu_{J}^{\lambda}\left(\left \{ x \mid |S_{J, r(J)}\mu_{J}^{\lambda}(x)|> k\frac{\mu_{J}^{\lambda}(J)}{|J|}\right \}\right)\geq \mu_{J}^{\lambda}(I(J)^{-})=
\tfrac{1}{6}\lambda 2^{-4k}|J|=\tfrac{1}{4}2^{-4k}\mu_{J}^{\lambda}(J),
$$
as desired.
\end{proof}

Actually this estimate \eqref{e.distestimate} is unimprovable,  as follows from the John-Nirenberg 
inequality.  (Our point of view in this construction is informed by the extension of the 
John-Nirenberg inequality in the weighted setting, as established in the 
work of the author with M. Lacey and S. Petermichl \cite{0906.1941}*{page 137}.)   

We have shown that we can construct a particular Haar multiplier, that with 
respect to Lebesgue measure has no cancellation, and we have reversed the  John-Nirenberg inequality. 
The success of this proof is based upon the observation 
that we can iterate this construction on the elements of the partition in \eqref{e.supp}.  
Namely, we are free to change the measure $ \mu _{J} ^{\lambda }$ provided we \emph{do not 
change the numbers}  $ \mu _{J} ^{\lambda } (I)$, for $ I\in \Xi _{J} ^{+}$.  
And so, we will change the definition of $  \mu _{J} ^{\lambda } $, without changing its total 
measure, in such a way that we carefully track the corona, so that we have \eqref{decompsigma}. 
This means that at a different threshold, and a different part of our Haar multiplier, we will have 
a second reversal of the John-Nirenberg inequality.  This construction will then have to be iterated 
many times, to overcome the exponential nature of the John-Nirenberg inequality.   
All of these considerations are incorporated into this proposition.

\begin{proposition}\label{p.mainestimate}
Let $k\geq 1$ be a fixed integer. There exist a family of random Haar multipliers $T^{k}$ and a weight $w_{k} \not\equiv 0$, $w_{k}\in L_{loc}^{1}$ such that 
\begin{equation}\label{e.mainestimate}
\sum_{L\in \mathcal L, L\subset [0,1]} w_{k}\left(\left\{ x\mid |T_{L}^{k}w_{k}(x)|>k\frac{w_{k}(L)}{|L|} \right\}\right ) \geq c w_{k}([0,1]),
\end{equation}
where $\mathcal L$ is the stopping collection in the corona decomposition associated to measure $w_{k}$ as defined in  \eqref{d.linearizing} and $c=1/6$.
\end{proposition}

For the proof we need this definition.

\begin{definition}\label{d.levelcorona}
Let $k\geq 1$ be an integer, $J$ be a dyadic interval and $\Xi_{J}$ be as above. We define the set of intervals $\mathcal L(J)$ as
$$
\mathcal L(J):=\left\{ L'(I):= I^{--} \mid I\in \Xi_{J}^{+}\right\}
$$
\end{definition}

Notice that the map 
\begin{align}\label{e.bijection}
\Psi\;:\;\, &\Xi_{J}^{+}  \longmapsto  \mathcal L(J)\\
&I \,\,\,\,\, \longmapsto L'(I) \notag
\end{align} 
is a bijection and $\displaystyle |L'(I)|=\tfrac{1}{4}|I|$. Moreover, given $J$ dyadic interval, 
\begin{equation}\label{e.disjointness}
I(J)\cap L'=\emptyset \quad \text{ for all } \, L'\in \mathcal L(J).
\end{equation}

\begin{remark}\label{r.}   The passage to the `left-left child' above helps us 
keep track of the corona.  We will rescale all of the measure assigned to $ I\in \Xi _J ^{+}$ 
to `right two-thirds' of $ I ^{--}$.  Now, $ \tfrac 14 \cdot \tfrac 23 = \tfrac 16$, so 
to preserve measure, we will need to multiply by $ 6$.  This explains the powers of $ 6$ 
that appear below. 
\end{remark}

\begin{proof} The proof follows from an
inductive procedure. Let $\mathcal L_{0}$, $\mu_{0}$ and $S^{0}$ be as follows, $\mathcal L_{0}=\{[0,1)\}$, $\mu_{0}:=\mu_{[0,1)}^{1}$ and $S^{0}:=S_{[0,1), r[0,1)}$. For a picture of the first stage see figure \ref{f1.} below. In general, for every $j\geq 1$ we define
\begin{eqnarray*}
\mathcal L_{j}&=& \bigcup_{L\in \mathcal L_{j-1}}\mathcal L (L)\\
\mu_{j}&=& \sum_{i=0}^{j-1} \sum_{L\in\mathcal L_{i}}\mu_{I(L)}^{6^{i}} + \sum_{L'\in\mathcal L_{j}} \mu_{L'}^{6^{j}}\\
S^{j}&=& S^{j-1}+ \sum_{L'\in\mathcal L_{j}}S_{L', r(L')}.
\end{eqnarray*}
For the proof of the proposition the selection of signs $r(L)$ is irrelevant. 
See figure \ref{f2.} for a descriptive drawing of the second stage of the construction.

\vspace{0.1in}

\begin{figure}[H]
\begin{tikzpicture}  
\draw[->]  (-.5,0) -- (7,0);  \draw[->](0,-.5) -- (0,1.5); 
\draw (.1,1) -- (-.1,1) node[left] {1};
\draw[thick] (2,.1) -- (2,-.1) node[below] {jp};
\draw (6,.2) -- (6,-.2) node[below] {1}; 
 \foreach \x in {1.875,2.25,3,6}  { \draw (\x,.2) -- (\x,-.2);}
\draw (2.1,.3) node {$ I_3$};
\foreach\a/\b in {2/2.65,1/4.5} {\draw (\b,.3) node {$ I^{+}_\a$};}
 \draw[thick] (2,1) -- (6,1);   \draw[thick] (0,0) -- (2,0);  
\end{tikzpicture} 
\caption{The first stage of the construction. The measure is held constant to the left 
of the jumping point, and on the interval $ I_3$. It recurses on the positive half of the other intervals: $ I_1^{+}, I_2^{+}$.}
\label{f1.}
\end{figure}
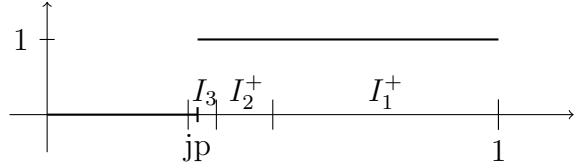

\begin{figure}[H]
\begin{tikzpicture}  
\draw[->]  (-.5,0) -- (7,0);  \draw[->](0,-.5) -- (0,2.5); 
\draw (6,.2) -- (6,-.2) node[below] {1}; 
\draw (.1,1) -- (-.1,1) node[left] {1}; \draw (.1,2) -- (-.1,2) node[left] {6}; 
\draw (.1,2) -- (-.1,2) node[left] {6};
 \foreach \x in {1.875,2.25,3,6}  { \draw (\x,.2) -- (\x,-.2);}
 
 \foreach \x  in {2,2.3125,2.4375,3.25,3.75} 
{ \draw[thick] (\x,.1) -- (\x,-.1);   } 

\draw[thick] (0,0) -- (2,0);   \draw[thick] (2,1) -- (2.25,1); 
\foreach \a/\b in {2.3125/2.4375,3.25/3.75}
{\draw[thick] (\a,2) -- (\b,2);} 

\draw[decorate,decoration=brace]  (3, .3) -- (3.75,.3) node[above] {$ I_1 ^{+--} $};

\end{tikzpicture} 

\caption{The second stage of the construction. Note that the horizontal scale is logarithmic. We repeat the first stage on the most left dyadic grandchild of each $I_{i}^{+}$, we jump to height $6$ so that the total measure of $I_{i}^{+}$ is preserved.}
\label{f2.}
\end{figure}
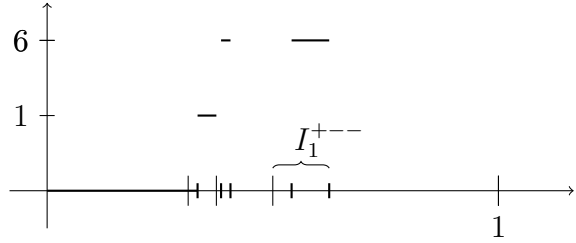


The next lemma states the main properties of the construction, namely, that the support of the measure built at each stage has shrunk with
respect to the previous one. At the same time the new measure preserves the total measure and the measure of the intervals playing a role in previous
stages. It is for this reason that we will refer to it as the `measure preserving lemma'.  Its proof is postponed until Section 5.

\begin{lemma}(Measure Preserving Lemma)\label{l.nose}
Let $\mathcal L_{j}$, $\mu_{j}$ and $S^{j}$ as above for any $j\geq 1$. Let $L\in \mathcal L_{j}$, we have the following estimates
\begin{enumerate}
\item $\sum_{L\in \mathcal L_{j}}|L |\le \frac{(1-2^{-4k})^{j}}{6^{j}}$,
\item $I(L)\cap L'=\emptyset$ for all $L'\in \mathcal L_{i}\,\, , L'\neq L\,\, ,  i\geq 1$. In particular $I(L)\cap I(L')=\emptyset$,
\item $\displaystyle \mu_{j+1}|_{[0,1]\backslash \cup_{L\in L_{j}} L}=\mu_{j}|_{[0,1]\backslash \cup_{L\in L_{j}} L}$,
\item $\mu_{j+1}(I)=\mu_{j+1}(L'(I))=\mu_{j}(I)$ for all $I\in \Xi^{+}_{L}$,
\item $\mu_{i}(I(L))=\mu_{j}(I(L))$ for all $i\geq j$,
\item $\mu_{i}(L)=\mu_{j}(L) \quad {\text{for all }} \quad i\geq j$,
\item $\mu_{i}\left(\left\{ |S_{L, r(L)}\mu_{i}(x)|>k\frac{\mu_{i}(L)}{|L|}\right\}\right)= \mu_{j}\left(\left\{ |S_{L, r(L)}\mu_{j}(x)|>k\frac{\mu_{j}(L)}{|L|}\right\}\right) \quad {\text{for all }} \quad i\geq j$.
\end{enumerate}
where in the last line $ S_{L,r(L)}$ is as in \eqref{e.haarmul}. 
\end{lemma}

We are now ready to prove the proposition. Let $M=M(k)>0$ be an integer that will depend on $k$ and will be chosen later. Let $T^{k}:=S^{M}$ and $w_{k}:=\mu_{M}$. We claim that
the corona decomposition of $\mathcal D'$ associated to $\mu_{l}$ is $\displaystyle (\mathcal L: \mathcal D'(L))$ where $\mathcal L=\bigcup_{i=0}^{l}
\mathcal L_{i}$ and $\mathcal D'$ is the set of dyadic intervals contained in $[0,1)$.  
The first stopping interval in the corona decomposition associated to $\mu_{0}$ is $[0,1)$. It is easy to see that this is actually the only stopping interval, therefore $\mathcal L_{0}$ is the stopping collection in the corona decomposition associated to $\mu_{0}$ and the claim is true in this case. The two facts that allow us to conclude the claim in general are: (a) parts (3) and (6) of Lemma \ref{l.nose}, which let us keep the
corona of the measure $\mu_{j}$ when we move to stage $j+1$. 
Using backwards induction we have proved $\bigcup_{i=0}^{l-1}\mathcal L_{i}$ forms part of the stopping intervals associated to $\mu_{l}$.

(b) For the next stage of the corona: due to parts (4) and (5) of Lemma \ref{l.nose} and \eqref{e.supp}, it is enough to consider the
children of $I\in \Xi_{L}^{+}$ for all $L\in \mathcal L_{l-1}$. We have these relations: \begin{eqnarray*} \mu_{l}(I^{+})&=&0\\
\frac{\mu_{l}(I^{-})}{|I^{-}|}&<& 4 \frac{\mu_{l}(L)}{|L|}\\ \mu_{l}(I^{-+})&=&0\\ \frac{\mu_{l}(I^{--})}{|I^{--}|}&\geq& 4
\frac{\mu_{l}(L)}{|L|}\\ \frac{\mu_{l}(J)}{|J|}&<& 4 \frac{\mu_{l}(I^{--})}{|I^{--}|}, \quad \text{for all } J\subset I^{--} \end{eqnarray*}
which describe the new level of stopping intervals as $\left \{L'(I) \mid I\in\Xi_{L}^{+},\,\, L\in \mathcal L_{l-1}\right\}$. It is easy to
check that this is the last stage. Herein we have proved that the corona decomposition associated to $w_{k}$ is $\displaystyle (\mathcal L:
\mathcal D'(L))$ with $\mathcal L=\bigcup_{i=0}^{M}\mathcal L_{i}$. 

\smallskip 

We now decompose the operator as described in \eqref{r.decompT}, $T^{k}= \sum_{L\in \mathcal L}T_{L}^{k}$.  Notice that $T_{L}^{k}=S_{L, r(L)}$ as  defined in \eqref{e.haarmul}. We finish the proof using parts (4), (6) and (7) of Lemma \ref{l.nose} together with Lemma \ref{l.distestimate}.
\begin{eqnarray*}
\sum_{\substack{L\in \mathcal L\\ L\subset [0,1]}} w_{k}\left(\left\{ x\mid |T_{L}^{k}w_{k}(x)|>k\frac{w(L)}{|L|} \right\}\right )&= &\sum_{i=0}^{M}\sum_{L\in \mathcal L_{i}}w_{k}\left(\left\{|S_{L, r(L)}w_{k}(x)|>k\frac{w_{k}(L)}{|L|}\right\}\right)\\
&=& \sum_{i=0}^{M}\sum_{L\in \mathcal L_{i}}\mu_{i}\left(\left\{|S_{L, r(L)}\mu_{i}(x)|>k\frac{\mu_{i}(L)}{|L|}\right\}\right)\\
& = &\sum_{i=0}^{M}\sum_{L\in \mathcal L_{i}}\mu_{L}^{6^{i}}\left(\left\{|S_{L, r(L)}\mu_{L}^{6^{i}}(x)|>k\frac{\mu_{L}^{6^{i}}(L)}{|L|}\right\}\right)\\
& \geq & \sum_{i=0}^{M}\sum_{L\in \mathcal L_{i}}\frac{1}{6} 6^{i}2^{-4k}|L|\\
& = & \tfrac{1}{6}2^{-4k} \sum_{i=0}^{M} (1-2^{-4k})^{i}\\
& = & \tfrac{1}{6}\left (1-(1-2^{-4k})^{M+1}\right )>\tfrac{1}{6}w_{k}([0,1]), 
\end{eqnarray*}
for $M=\left[\frac{\log(3)}{\log(1-2^{-4k})^{-1}}\right]+1$. Note that $M$ behaves exponentially with respect to $k$, which is exactly what one should expect from \eqref{e.distestimate}.

\end{proof}

\section{Proof of Main Theorem} 

Using Proposition \ref{p.dual} and \eqref{e.dual}, we can reduce the proof of the main theorem to finding a weight $w$, a Haar multiplier
$T$ and a function $f\in L^{2}(w)$ such that for all $C>0$
\begin{equation*}
\int_{\mathbb R}|T(fw)|^{2}\sigma(x)dx\geq C\ \int_{\mathbb R}|f|^{2}w(x)dx. 
\end{equation*}
There is one more reduction, we can use a gliding hump argument to deduce the infinitary  
inequality above from the following finitary one,

\begin{lemma}(Main Lemma) \label{l.main}
For all $k\ge 1$, there exist a Haar multiplier $T^{k}$ and a weight $w_{k} \not\equiv 0$, such that
\begin{equation}\label{e.mainl}
\int_{[0,1)}|T^{k}\left(w_{k} 1_{[0,1)}\right)(x)|^{2}\sigma_{k}(x)dx\geq  Ck^{2}w_{k}([0,1)),
\end{equation}
where $C$ is a universal constant.
\end{lemma}

\begin{proof}
 Let $k\geq 1$ be a fixed natural number. Define $T^{k}$ and $w_{k}$ as in Proposition \ref{p.mainestimate} and let $\displaystyle(\mathcal L: \mathcal D( L))$ be the corona decomposition associated to $w_{k}$. We now use the decomposition of $T^{k}$ and $\sigma_{k}$ as suggested in \eqref{r.decompT} and \eqref{decompsigma} respectively. It is now time to determine the more convenient choices of signs $r_L$ that appear in the definition of \eqref{e.haarmul}. By Khintchine's inequalities we can find a sequence of signs $\{r_L\}_{L\in\mathcal L}$, $r_L\in \{+1, -1\}$ so that the first inequality below holds. That together with Chebyshev's inequality and
Proposition \ref{p.mainestimate} provide the desired estimate.
\begin{eqnarray*}
\int_{[0,1)}|T^{k}\left(w_{k} 1_{[0,1)}\right)(x)|^{2}\sigma_{k}(x)dx &\geq& 
\sum_{L\in \mathcal L}\int_{[0,1)}|T^{k}_{L}\left(w_{k} 1_{[0,1)}\right)(x)|^{2}\sigma_{k}(x)dx\\
 & \geq & \sum_{L\in \mathcal L}\left(\frac{w_{k}(L)}{|L|}\right)^2 k^{2}\sigma_{k}\left(\left\{ T_{L}^{k}w_{k}>k\frac{w_{k}(L)}{|L|}\right\}\right)\\
 & \geq & \frac{k^{2}}{64} \sum_{L\in \mathcal L}w_{k}\left(\Delta_{1}\left(\left\{ T_{L}^{k}w_{k}>k\frac{w_{k}(L)}{|L|}\right\}\right) \right)\\
 & \geq & \frac{k^{2}}{64} \sum_{L\in \mathcal L}w_{k}\left(\left\{ T_{L}^{k}w_{k}>k\frac{w_{k}(L)}{|L|}\right\}\right)\geq Ck^{2}w_{k}([0,1)),\\
\end{eqnarray*}
where $C$ is a universal constant. In the last line, we have used \eqref{e.mineqlevelset} and property (2) of Lemma \ref{l.nose}.
\end{proof}

\section{Proof of the measure preserving lemma}

\begin{proof} Let us start proving (1).  A backwards induction argument allows us to reduce the problem to proving that given $L\in L_{j-1}$ and for any $j\geq 1$,
\begin{equation}\label{e.stoppinglebesgue}
\sum_{L'\subset L, L'\in \mathcal L_{j}} |L'|= \frac{1}{6}(1-2^{-4k})|L|.
\end{equation}
The proof of \eqref{e.stoppinglebesgue} uses $|L'(I)|=\frac{1}{4}|I|$, with $L'(I)$ defined in \eqref{e.bijection}. 
\begin{eqnarray*}
\sum_{L'\subset L, L'\in \mathcal L_{j}} |L'|& = & \frac{1}{4}|[\textup{rep}(I(L)), \textup{rep}(L)]|\\
 &=& \frac{1}{4}\left[\frac{2}{3}|L|-\frac{2}{3}|I(L)|\right]\\
 &=& \frac{1}{6}(1-2^{-4k})|L|.
\end{eqnarray*}

Proof of (2). Let $L,L'\in \mathcal L$, we need to distinguish two cases.
\begin{enumerate}
\item If $L\cap L'=\emptyset$ then $I(L)\cap L'=\emptyset$ trivially.
\item Suppose $L'\subset L$, then there exist $i$ and $j$, $i>j$ such that $L'\in\mathcal L_{i}$ and $L\in\mathcal L_{j}$.
We can find $\tilde{L}\in \mathcal L(L)$ such that $L'\subset \tilde{L}$, then $I(L)\cap \tilde{L}=\emptyset$ by \eqref{e.disjointness}. Therefore $I(L)\cap L'=\emptyset$ as desired.
\end{enumerate}

Conclusion (3) follows trivially from the definition of the measures $\mu_{j}$ and $\mu_{j+1}$.  

Next we are proving (4). Let $I\in\Xi_{L}^{+}$ and let $L'(I)$ be as
 in \eqref{e.bijection}, then
$$
\mu_{j+1}(L'(I))= \mu_{L'(I)}^{6^{j+1}}(L'(I))=6^{j+1}\frac{2}{3}|L'(I)|=6^{j+1}\frac{2}{3}\frac{1}{4}|I|=6^{j}|I|=\mu_{L}^{6^{j}}(I)=\mu_{j}(I),
$$
where we have used the definition of $\mu_{j}$ and the fact that $\textup{jp}(L)\notin I$. We want to make one more comment, estimate (4) is the heart of the measure preserving property.

Proof of (5).
 Let $L\in\mathcal L_{j}$, then $I(L)\cap L'=\emptyset$ for all $L'\neq L$ by property (2). The following estimates conclude the proof,
\begin{eqnarray*}
\mu_{j}(I(L)) &=\mu_{L}^{6^{j}}(I(L))&=6^{j}\frac{2}{3}|I(L)|\\
\mu_{i}(I(L)) &=\mu_{I(L)}^{6^{j}}(I(L))&=6^{j}\frac{2}{3}|I(L)|.
\end{eqnarray*}

The proof of (6) can be deduced from the following equality and a backwards induction argument,
\begin{equation}\label{e.nochangemeasure}
\mu_{j+1}(L)= \mu_{j}(L)\quad {\text{ for all }}j\geq 0 {\text{ and for all }} L\in \mathcal L_{j}.
\end{equation} 
We complete the proof of \eqref{e.nochangemeasure} using (4) and \eqref{e.supp}. Let $L\in \mathcal L_{j}$, then $I(\tilde{L})\cap L=\emptyset$
 for all $\tilde{L}\in\mathcal L_{i}$, $i<j$ by property (2)  and we get the following
\begin{eqnarray*}
\mu_{j+1}(L)&=& \mu_{I(L)}^{6^{j}}(L) + \sum_{L'\in \mathcal L_{j+1}, L'\subset L}\mu^{6^{j+1}}_{L'}(L)\\
 & =& \mu_{I(L)}^{6^{j}}(L) + \sum_{L'\in \mathcal L_{j+1}, L'\subset L}\mu^{6^{j+1}}_{L'}(L')\\
  & =& \mu_{I(L)}^{6^{j}}(L) + \sum_{I\in \Xi^{+}_{L}}\mu_{L}^{6^{j}}(I)\\
  &=& \mu_{L}^{6^{j}}(L)=\mu_{j}(L).\\
\end{eqnarray*}

We now turn to proving (7). Let $L\in\mathcal L_{j}$, again a backwards induction argument reduces (7) to prove
\begin{equation*}
\mu_{j+1}\left(\left\{\left| S_{L, r(L)}\mu_{j+1}\right|>k\frac{\mu_{j+1}(L)}{|L|}\right\}\right)= \mu_{j}\left(\left\{ \left| S_{L, r(L)}\mu_{j}\right|>k\frac{\mu_{j}(L)}{|L|}\right\}\right).
\end{equation*}
The strategy will be to verify $S_{L, r(L)}\mu_{j+1}=S_{L, r(L)}\mu_{j}$. The rest of the proof follows from \eqref{e.mineqlevelset} and properties (5) and (6). This said, we are going to prove
\begin{equation}\label{e.eproduct}
\langle \mu_{j+1},h_{I}\rangle = \langle \mu_{j}, h_{I}\rangle \quad {\text{ for all }} I\in \Xi_{L}.
\end{equation}
Suppose $I=I(L)$, then $\mu_{j+1}|_{I(L)}=\mu_{j}|_{I(L)}$ proving \eqref{e.eproduct} for this particular case. Suppose $I\in \Xi_{L}$ but $I\neq I(L)$, then $I^{+}\in \Xi_{L}^{+}$ and $I^{-}=\left( I^{-}\backslash \bigcup_{J\in\Xi_{L}^{+}, J\subset I^{-}}J \right)\cup \bigcup_{J\in\Xi_{L}^{+}, J\subset I^{-}}J$.
The decomposition of $I^{+}$ and $I^{-}$ together with property (4)   proves the desired estimate 
\begin{eqnarray*}
\langle \mu_{j+1},h_{I}\rangle&=& \mu_{j+1}(I^{+})-\mu_{j+1}(I^{-})\\
&=& \mu_{j+1}(L'(I^{+}))- \mu_{j+1}(I^{-}\backslash \cup_{J\in\Xi_{L}^{+}, J\subset I^{-}}J)-\sum_{J\in\Xi_{L}^{+}, J\subset I^{-}}\mu_{j+1}(J)\\
&=& \mu_{j}(I^{+})-\mu_{j}(I(L))- \sum_{J\in\Xi_{L}^{+}, J\subset I^{-}}\mu_{j}(J)\\
&=& \mu_{j}(I^{+})-\mu_{j}(I^{-})= \langle \mu_{j}, h_{I}\rangle.
\end{eqnarray*}
\end{proof}


 
\end{document}